\documentclass[10pt,a4paper]{article}
\usepackage[a4paper,hmargin={1.5cm,1.5cm},vmargin={2cm,2cm}]{geometry}
\usepackage[utf8]{inputenc}
\usepackage[affil-sl]{authblk}

\usepackage[backend=biber,style=ieee,citestyle=numeric-comp,sorting=nty]{biblatex}
\usepackage[dvipsnames]{xcolor}
\usepackage[singlelinecheck=false]{caption}
\usepackage[linesnumbered,ruled]{algorithm2e}
\usepackage[colorlinks=true,linkcolor=Rhodamine,citecolor=Rhodamine]{hyperref}

\usepackage{url}
\usepackage{hyphenat}
\usepackage{parskip}

\usepackage{tabularray}

\usepackage{amsmath}
\usepackage{upgreek}
\usepackage{breqn}
\usepackage{mathtools}
\usepackage{amsfonts}
\usepackage{amscd}
\usepackage{amssymb}

\usepackage{tikz}
\usetikzlibrary{positioning,arrows.meta,fit}
\usepackage{graphicx}
\usepackage{pifont}
\usepackage{mathrsfs}
\usepackage{fancyhdr}
\usepackage{subfiles}

\usepackage{subfigure}

\usepackage{amsthm}
\theoremstyle{definition}
\newtheorem{theorem}{Theorem}
\newtheorem{definition}{Definition}
\newtheorem{remark}{Remark}
\newtheorem{corollary}{Corollary}
\newtheorem{lemma}{Lemma}

\newtheorem{example}{Example}

\definecolor{primary}{HTML}{6200ea}
\definecolor{secondary}{HTML}{00e5ff}

\begin{document}

\font\titleFont=cmr12 at 17pt
\title{{\titleFont A bipolar fuzzy relation equation framework for clinical decision support systems}}

\def\correspondingauthor{\footnote{Corresponding Author: a.ghodousian@ut.ac.ir (A. Ghodousian)}}

\author[1]{Amin Ghodousian \correspondingauthor{}}
\author[2]{Mohammad Sedigh Chopannavaz}

\affil[1]{Faculty of Engineering Science, College of Engineering, University of Tehran, P.O.Box 11365-4563, Tehran, Iran.}
\affil[2]{Department of Engineering Science, College of Engineering, University of Tehran, Tehran, Iran.}

\date{} 

\makeatletter
\renewenvironment{proof}[1][\proofname]{%
    \par\pushQED{\qed}\normalfont%
    \topsep6\p@\@plus6\p@\relax
    \trivlist\item[\hskip\labelsep\bfseries#1\@addpunct{.}]%
    \ignorespaces
}{%
    \popQED\endtrivlist\@endpefalse
}
\makeatother

\maketitle


\section*{Abstract}\label{sec_abs}
Clinical decision support frequently involves heterogeneous measurements, expert assessments, and guideline-based knowledge whose relationships to clinical decisions are gradual rather than purely binary. This paper develops a clinical decision-support framework based on bipolar fuzzy relational optimization, in which favorable and unfavorable clinical relationships are represented jointly through positive and negative fuzzy relational matrices together with a fuzzy requirement vector. The simultaneous participation of each decision variable and its complement in the relational constraints enables bipolar clinical information to be incorporated within a unified optimization model. The minimum \(t\)-norm is adopted to provide a noncompensatory, bottleneck-type interpretation of the interaction between fuzzy relational grades and decision levels. We characterize the feasible recommendation set by deriving clinical admissibility intervals and effective evidence activation sets, which provide the basis for feasibility analysis and systematic reduction of the relational system. The complete feasible set is then represented as a finite union of clinical recommendation regions associated with admissible Clinical Evidence-Assignment Functions. For a continuous objective function that is coordinatewise monotone with respect to the decision variables, a region-wise optimal candidate is constructed for each admissible evidence assignment, and comparison of the resulting finite collection of candidates yields a globally optimal recommendation. An algorithm integrating feasibility analysis, system reduction, region-wise optimization, and global selection is developed and illustrated through a numerical CDSS example. The results show that bipolar fuzzy relational optimization provides a structured and mathematically rigorous mechanism for transforming fuzzy clinical relational information into feasible and optimized decision-support recommendations.

\textbf{Keywords}: Bipolar fuzzy relational equations, Clinical decision support, Fuzzy relational optimization, Max--min composition, Global optimization, Decision making.


\section{Introduction}

Clinical decision-making is rarely based on information that is completely
certain, homogeneous, and binary. Patient measurements, expert assessments,
clinical guidelines, and treatment-related knowledge may all contribute to a
decision, but their relationships to the final recommendation are often gradual.
Fuzzy set theory provides a natural language for representing such graded
relationships, while fuzzy relational equations (FREs) provide a mathematical
mechanism for linking fuzzy inputs, relational knowledge, and required output
levels. When optimization is imposed over the solution set of an FRE, the
relational model becomes not only a representation of uncertain knowledge but
also a structured feasible region over which a preferable decision can be
identified.

The connection between fuzzy relational modelling and medicine is historically
well established. Sanchez~\cite{ref_39} introduced composite fuzzy relation
equations and applied them to medical diagnosis, where fuzzy relations were
used to connect patients with clinical propositions such as symptoms and
diagnoses. Umeyama~\cite{ref_49} subsequently considered complementary
diagnostic information through negated propositions. These early developments
are important for the present study because they show that fuzzy relational
models were motivated from the beginning by the need to represent uncertain
clinical knowledge rather than by purely abstract algebraic considerations.

The mathematical theory of FREs was subsequently developed in several
directions. Di Nola et al.~\cite{ref_06} provided a systematic treatment of
fuzzy relation equations for fuzzy modelling, and Pedrycz~\cite{ref_35}
studied generalized forms of FREs. Later investigations addressed solvability,
minimal and maximal solutions, structural characterizations, and computational
resolution under a variety of relational compositions
\cite{ref_02,ref_03,ref_28,ref_29,ref_37,ref_41,ref_43,ref_44}. A characteristic
feature of many FRE systems is that their feasible sets are non-convex and may
be described through extremal solutions or finite families of solution regions.
This structure is particularly relevant when an objective function is added,
because standard convex-optimization techniques are not generally applicable
directly.

Accordingly, a substantial literature has considered optimization problems
subject to fuzzy relational constraints. Linear optimization has been studied
for max--min, max--product, and more general max--$t$-norm compositions
\cite{ref_01,ref_11,ref_21,ref_22,ref_38,ref_42,ref_46}. Fang and
Li~\cite{ref_11}, for example, transformed a max--min FRE-constrained linear
problem into an integer-programming formulation, while Chang and
Shieh~\cite{ref_01} and Wu and Guu~\cite{ref_46} developed structural and
reduction results. Further extensions considered nonlinear objectives
\cite{ref_16,ref_17}, linear-fractional objectives~\cite{ref_48}, fuzzy
coefficients~\cite{ref_05}, interval-valued or intuitionistic fuzzy
information~\cite{ref_07}, and fuzzy relational inequalities
\cite{ref_14,ref_18,ref_19,ref_25,ref_52,ref_56}. Collectively, these studies
established FREs and FRIs as a bridge between fuzzy relational modelling and
optimization-based decision support.

A further conceptual step is required when a relational system contains two
opposing aspects. In many decision environments, the same decision component
may participate through a favorable relation in one direction and through an
unfavorable or complementary relation in the other. Bipolar representations
provide a systematic way to distinguish these aspects of information and
preference~\cite{ref_08,ref_09}. Within a bipolar fuzzy relational equation
(BFRE), this idea is expressed by allowing both a decision variable $x_j$ and
its complement $1-x_j$ to appear in the same relational constraint.

Freson, De Baets, and De Meyer~\cite{ref_12} introduced bipolar max--min
constraints in a linear optimization setting. Their formulation explicitly
coupled $x_j$ and $1-x_j$ and illustrated the model through degrees of
appreciation and disappreciation. The theory was subsequently developed for
max--min resolution~\cite{ref_26}, max--product compositions
\cite{ref_04,ref_58}, and max--\L{}ukasiewicz optimization and resolution
\cite{ref_27,ref_30,ref_50,ref_53}. These studies showed that the structure of
a BFRE feasible set depends strongly on the selected relational composition.
The framework was later extended to continuous strict $t$-norms~\cite{ref_60},
continuous Archimedean $t$-norms~\cite{ref_69}, and ultimately arbitrary
continuous $t$-norms~\cite{ref_63}. In particular, the generalized formulation
in~\cite{ref_63} considers nonlinear programming subject to
$
A^{+}\varphi x\vee A^{-}\varphi(\mathbf{1}-x)=b,
~ x\in[0,1]^n,
$
with $\varphi$ an arbitrary continuous $t$-norm, and establishes feasibility
conditions, a complete representation of the feasible set, simplification
rules, and a finite-candidate strategy for global optimization.

The broader application literature shows that fuzzy relational models are not
restricted to theoretical examples. FREs and FRIs have been applied to fault
diagnosis~\cite{Rotshtein2006}, image and video processing~\cite{Loia2005},
resource allocation~\cite{Xiao2019}, wireless communication management
\cite{Yang2016,Yang2018}, peer-to-peer systems~\cite{Yang2018a}, and
supply-chain decision-making~\cite{Lin2019,Lei2025}. The medical domain is
particularly natural because the original Sanchez framework already treated
clinical relations as fuzzy, and later medical fuzzy systems have continued to
use graded knowledge representations in computational decision support
\cite{HernandezJulio2019,Improta2020}. These developments motivate the use of
BFRE optimization as a structured mechanism for processing fuzzy clinical
relational information and generating feasible and preferred recommendations.

In a practical fuzzy decision-support setting, the relational coefficients used
by the BFRE model are naturally obtained after the relevant clinical information
has been expressed in fuzzy form. Quantitative measurements, linguistic
assessments, guideline-based knowledge, and expert judgments can be mapped by an
appropriate fuzzification or fuzzy knowledge-representation procedure into
grades in $[0,1]$. These grades populate the positive and negative fuzzy
relational matrices $A^{+}$ and $A^{-}$ and the fuzzy requirement vector $b$,
which therefore represent graded clinical relationships and requirements rather
than raw clinical measurements. This organization follows the usual fuzzy-system
principle in which membership-based representation provides the fuzzy
information subsequently processed by the reasoning or decision model
\cite{HernandezJulio2019,Improta2020}. In the present framework, the BFRE stage
uses these fuzzy relational inputs to characterize feasibility and determine an
optimal recommendation vector.

\subsection{BFRE formulation of the clinical decision-support problem}\label{sec-clinical-model}

Within the BFRE layer, let $T_1,\ldots,T_m$ denote candidate treatment
alternatives and $C_1,\ldots,C_n$ the clinical criteria entering the relational
decision model. The favorable and unfavorable fuzzy relationships between
treatment $T_i$ and criterion $C_j$ are represented by
$
A^{+}=\left(a_{ij}^{+}\right)_{m\times n}$,
$A^{-}=\left(a_{ij}^{-}\right)_{m\times n}.
$
Here $a_{ij}^{+}$ is the fuzzy degree of a positive or supportive relation and
$a_{ij}^{-}$ is the fuzzy degree of the corresponding negative, unfavorable,
or complementary relation.

The BFRE layer determines a vector
$
x=(x_1,\ldots,x_n)^T\in[0,1]^n,
$
where $x_j$ is the decision/recommendation level associated with criterion
$C_j$ in the relational model. Its complement
$1-x_j$ participates simultaneously in the bipolar constraint, giving the
standard bipolar decision structure of the BFRE.

For the clinical application developed here, the minimum $t$-norm is selected,
$
\varphi(u,v)=\min\{u,v\}.
$
The choice has a noncompensatory bottleneck interpretation: the effective
relational contribution cannot exceed either the fuzzy relationship grade or
the associated decision level. Thus
$
\varphi(a_{ij}^{+},x_j)=\min\{a_{ij}^{+},x_j\}$,
$\varphi(a_{ij}^{-},1-x_j)=\min\{a_{ij}^{-},1-x_j\}.
$
Let
$
b=(b_1,\ldots,b_m)^T\in[0,1]^m
$
denote the fuzzy vector of required relational degrees associated with the
candidate-treatment requirements. The resulting optimization problem is
\begin{equation}\label{eq_1}
    \begin{array}{ll}
        \min & f(x) \\[1mm]
        \text{s.t.}
        & \min (A^{+}, x)
        \vee
        \min (A^{-}, (\mathbf{1}-x))=b,\\
        & x\in[0,1]^n,
    \end{array}
\end{equation}
where $\mathbf{1}$ denotes the all-ones vector. The objective $f(x)$ expresses
how feasible recommendation configurations are compared. Depending on the
application, it may encode preference, risk, treatment burden, cost, resource
use, or another decision criterion.
The entries satisfy
$
0\le a_{ij}^{+}\le1$,
$0\le a_{ij}^{-}\le1$,
$0\le b_i\le1,
$
for every
$i\in\mathscr{I}=\{1,\ldots,m\}$ and
$j\in\mathscr{J}=\{1,\ldots,n\}$.
For each treatment alternative $T_i$, the corresponding relational equality in
Problem~\eqref{eq_1} is
\begin{equation}\label{eq-2}
    \max_{j=1}^{n}
    \left\{
    \max\left\{
    \min\left(a_{ij}^{+},x_j\right),
    \min\left(a_{ij}^{-},1-x_j\right)
    \right\}
    \right\}
    =b_i.
\end{equation}
The inner maximum expresses the bipolar character of criterion $C_j$: either
its positive branch or its complementary negative branch may provide the active
relational contribution. The outer maximum is the native max-composition of the
BFRE and identifies the strongest active criterion for the $i$-th relational
requirement. Equation~\eqref{eq-2} therefore should be read as a fuzzy
relational requirement rather than as a probabilistic accumulation of clinical
evidence. The equality requires the strongest admissible bipolar contribution
to attain exactly the prescribed fuzzy degree $b_i$.

The feasible recommendation set is
$
S(A^{+},A^{-},b)
=
\left\{
x\in[0,1]^n:
A^{+}\varphi x
\vee
A^{-}\varphi(\mathbf{1}-x)=b
\right\}.
$
Every element of this set is a relational decision configuration compatible
with all fuzzy requirements encoded by $A^{+}$, $A^{-}$, and $b$. The
optimization stage then selects the configuration preferred by $f$.

The contribution of the paper can therefore be summarized as follows. First,
for the max--min clinical BFRE we derive componentwise admissibility and
activation structures that permit systematic feasibility analysis. Second, the
complete feasible recommendation set is represented through a finite family of
regions induced by admissible Clinical Evidence-Assignment Functions. Third,
for continuous coordinatewise monotone objectives, a finite candidate set is
constructed and shown to contain a global optimum. Finally, these components are
assembled into an algorithm and demonstrated through a numerical CDSS example.

The remainder of the paper is organized as follows. Section~\ref{sec-2}
develops the feasibility analysis and introduces clinical admissibility
intervals and effective evidence activation sets. Section~\ref{sec-3}
characterizes the feasible recommendation set through admissible Clinical
Evidence-Assignment Functions. Section~\ref{sec-5} develops the local and
global optimization results and establishes the finite-candidate solution
procedure.

\section{Feasibility Analysis of the Clinical Decision Support System}\label{sec-2}

The feasibility analysis identifies the recommendation profiles that are
compatible with the clinical criteria represented in the CDSS. We first
consider the bipolar relation associated with a treatment alternative
$T_i$ and a clinical criterion $C_j$,
$
\max\left\{
\min\left(a_{ij}^{+},x_j\right),
\min\left(a_{ij}^{-},1-x_j\right)
\right\}=b_i,
~ i\in\mathscr{I},\ j\in\mathscr{J}.
$
For each pair $(i,j)$, the feasible values of $x_j$ are characterized
through the positive and negative clinical relationships between
$T_i$ and $C_j$. These pairwise characterizations are then combined
over the clinical criteria to determine the feasible recommendation
set of the complete CDSS model.

For each $i\in\mathscr{I}$, let $S_i$ denote the set of recommendation
vectors satisfying the clinical requirement associated with treatment
alternative $T_i$, that is,
$
S_i=
\left\{
x\in[0,1]^n:
\max_{j=1}^{n}
\left\{
\max\left\{
\min\left(a_{ij}^{+},x_j\right),
\min\left(a_{ij}^{-},1-x_j\right)
\right\}
\right\}
=b_i
\right\}.
$
Thus, $S_i$ collects all recommendation profiles for which the
bipolar clinical relationships associated with $T_i$ attain the
required degree $b_i$. Let $S(A^{+},A^{-},b)$ denote the feasible
recommendation set of Problem~(\ref{eq_1}). Then,
$
S(A^{+},A^{-},b)
=
\bigcap_{i\in\mathscr{I}}S_i.
$
Hence, a recommendation profile belongs to the feasible set of the
CDSS if and only if it satisfies the prescribed clinical requirements
for all treatment alternatives.

\begin{definition}\label{def-2}
	For each $i\in\mathscr{I}$ and each $j\in\mathscr{J}$, define
	$
	S_{ij}^{+}
	=
	\left\{
	x_j\in[0,1]:
	\min\left(a_{ij}^{+},x_j\right)=b_i
	\right\}$,
	$S_{ij}^{-}
	=
	\left\{
	x_j\in[0,1]:
	\min\left(a_{ij}^{-},1-x_j\right)=b_i
	\right\}.
	$
	The sets $S_{ij}^{+}$ and $S_{ij}^{-}$ are referred to as the
	positive and negative evidence activation sets, respectively. They
	contain the recommendation levels of criterion $C_j$ at which the
	corresponding positive or negative clinical relationship with
	treatment $T_i$ attains the required degree $b_i$.
	
	Furthermore, define
	$
	I_{ij}^{+}
	=
	\left\{
	x_j\in[0,1]:
	\min\left(a_{ij}^{+},x_j\right)\leq b_i
	\right\},
	~
	I_{ij}^{-}
	=
	\left\{
	x_j\in[0,1]:
	\min\left(a_{ij}^{-},1-x_j\right)\leq b_i
	\right\}.
	$
	The sets $I_{ij}^{+}$ and $I_{ij}^{-}$ are referred to as the
	positive and negative clinical admissibility sets, respectively.
	They contain the recommendation levels of criterion $C_j$ for which
	the corresponding positive or negative clinical relationship with
	treatment $T_i$ does not exceed the required degree $b_i$.
\end{definition}

According to Definition~\ref{def-2}, we immediately have
$
S_{ij}^{+}\subseteq I_{ij}^{+},
~
S_{ij}^{-}\subseteq I_{ij}^{-},
~
\forall i\in\mathscr{I},\ \forall j\in\mathscr{J}.
$
For the minimum $t$-norm, the sets introduced in
Definition~\ref{def-2} admit the following explicit forms:
\begin{equation}\label{eq-4}
	\begin{alignedat}{4}
		S_{ij}^{+}
		&=
		\begin{cases}
			\{b_i\}, & a_{ij}^{+}\geq b_i,\\
			\varnothing, & a_{ij}^{+}<b_i,
		\end{cases}
		&
		,~
		I_{ij}^{+}
		&=
		\begin{cases}
			[0,b_i], & a_{ij}^{+}\geq b_i,\\
			[0,1], & a_{ij}^{+}<b_i,
		\end{cases}
		&
		,~
		S_{ij}^{-}
		&=
		\begin{cases}
			\{1-b_i\}, & a_{ij}^{-}\geq b_i,\\
			\varnothing, & a_{ij}^{-}<b_i,
		\end{cases}
		&
		,~
		I_{ij}^{-}
		&=
		\begin{cases}
			[1-b_i,1], & a_{ij}^{-}\geq b_i,\\
			[0,1], & a_{ij}^{-}<b_i.
		\end{cases}
	\end{alignedat}
\end{equation}

\begin{definition}\label{def-3}
	For each $i\in\mathscr{I}$ and each $j\in\mathscr{J}$, define
	$
	S_{ij}
	=
	\left\{
	x_j\in[0,1]:
	\max\left\{
	\min\left(a_{ij}^{+},x_j\right),
	\min\left(a_{ij}^{-},1-x_j\right)
	\right\}
	=b_i
	\right\}
	$
	and
	$
	I_{ij}
	=
	\left\{
	x_j\in[0,1]:
	\max\left\{
	\min\left(a_{ij}^{+},x_j\right),
	\min\left(a_{ij}^{-},1-x_j\right)
	\right\}
	\leq b_i
	\right\}.
	$
	The set $S_{ij}$ is called the \emph{bipolar evidence activation set},
	as it identifies the recommendation levels of criterion $C_j$ at which
	the combined positive and negative clinical relationships associated
	with treatment $T_i$ attain the required degree $b_i$. The set $I_{ij}$
	is called the \emph{bipolar clinical admissibility set}; it contains the
	recommendation levels for which the combined bipolar relationship does
	not exceed the required degree $b_i$.
\end{definition}

The following lemma establishes the relationship between the bipolar
admissibility and activation sets introduced in Definition~\ref{def-3}
and their positive and negative counterparts in
Definition~\ref{def-2}. These relations will be used in the subsequent
characterization of the feasible recommendation levels.

\begin{lemma}\label{lm-2}
	For each $i\in\mathscr{I}$ and $j\in\mathscr{J}$;
	\begin{itemize}
		\item [\textbf{(a)}]
			The bipolar clinical admissibility set is given by
			$
			I_{ij}=I_{ij}^{+}\cap I_{ij}^{-}.
			$
			Thus, a recommendation level of criterion $C_j$ is admissible for the
			bipolar relationship with treatment $T_i$ if and only if it is
			admissible with respect to both its positive and negative clinical
			relationships.
		\item [\textbf{(b)}]
			The bipolar evidence activation set is given by
			$
			S_{ij}=I_{ij}\cap
			\left(S_{ij}^{+}\cup S_{ij}^{-}\right).
			$
			Thus, within the bipolar clinical admissibility set, a recommendation
			level activates the relationship when at least one of the positive or
			negative clinical relationships attains the required degree $b_i$.
	\end{itemize}
\end{lemma}

\begin{proof}
	\textbf{(a)}
	The result follows directly from Definitions~\ref{def-2} and
	\ref{def-3}, since the Bipolar Clinical Admissibility Set is obtained
	by simultaneously satisfying the positive and negative admissibility
	conditions.
	\textbf{(b)}
	By Definition~\ref{def-3}, a recommendation level $x_j$ belongs to
	$S_{ij}$ if and only if
	$
	\min\left(a_{ij}^{+},x_j\right)\leq b_i
	$
	and
	$
	\min\left(a_{ij}^{-},1-x_j\right)\leq b_i,
	$
	with at least one of these two inequalities holding with equality.
	The two inequalities are equivalent to
	$
	x_j\in I_{ij}^{+}\cap I_{ij}^{-}=I_{ij}.
	$
	Moreover,
	$
	\min\left(a_{ij}^{+},x_j\right)=b_i
	$
	holds if and only if $x_j\in S_{ij}^{+}$, while
	$
	\min\left(a_{ij}^{-},1-x_j\right)=b_i
	$
	holds if and only if $x_j\in S_{ij}^{-}$.
	Therefore, at least one of the two bipolar relationships attains the
	required clinical degree $b_i$ if and only if
	$
	x_j\in S_{ij}^{+}\cup S_{ij}^{-}.
	$
	Combining the admissibility requirement with this attainment condition
	gives
	$
	S_{ij}
	=
	I_{ij}\cap
	\left(S_{ij}^{+}\cup S_{ij}^{-}\right).
	$
	Hence, the Bipolar Evidence Activation Set consists precisely of those
	recommendation levels that remain clinically admissible and for which
	the required degree $b_i$ is attained through either the positive or
	the negative clinical relationship.
\end{proof}

Based on~(\ref{eq-4}) and Lemma~\ref{lm-2}, the bipolar clinical
admissibility and activation sets can be explicitly characterized
according to the relative strengths of the positive and negative
clinical relationships. The following corollary summarizes the four
possible cases.

\begin{corollary}\label{corl-1}
	Suppose that $i\in\mathscr{I}$ and $j\in\mathscr{J}$. Then the
	bipolar clinical admissibility set $I_{ij}$ and the bipolar evidence
	activation set $S_{ij}$ are given by the following cases:
	\begin{itemize}
		\item[\textbf{(a)}]
		If $a_{ij}^{+}<b_i$ and $a_{ij}^{-}<b_i$, then
		$
		I_{ij}=[0,1],
		~
		S_{ij}=\varnothing.
		$
		Thus, neither the positive nor the negative clinical relationship
		associated with $T_i$ and $C_j$ is sufficiently strong to attain the
		required clinical degree $b_i$.
		\item[\textbf{(b)}]
		If $a_{ij}^{+}\geq b_i$ and $a_{ij}^{-}<b_i$, then
		$
		I_{ij}=[0,b_i],
		~
		S_{ij}=\{b_i\}.
		$
		In this case, the positive clinical relationship can attain the
		required degree, while the negative relationship cannot. Consequently,
		the admissible recommendation levels are bounded above by $b_i$.
		\item[\textbf{(c)}]
		If $a_{ij}^{-}\geq b_i$ and $a_{ij}^{+}<b_i$, then
		$
		I_{ij}=[1-b_i,1],
		~
		S_{ij}=\{1-b_i\}.
		$
		Here, only the negative clinical relationship can attain the required
		degree. Hence, the admissible recommendation levels are bounded below
		by $1-b_i$.
		\item[\textbf{(d)}]
		If $a_{ij}^{+}\geq b_i$ and $a_{ij}^{-}\geq b_i$, then
		$
		I_{ij}=[1-b_i,b_i],
		$
		and
		$
		S_{ij}
		=
		\begin{cases}
			\{1-b_i\}\cup\{b_i\}, & b_i>\dfrac{1}{2}\\
			[1-b_i,b_i], & b_i\leq\dfrac{1}{2}
		\end{cases}
		$.
		In this case, both clinical relationships are sufficiently strong to
		attain the required degree. The resulting admissible range is
		determined jointly by the positive and negative relationships, while
		the activation set depends on the relative position of the two
		activation levels $1-b_i$ and $b_i$.
	\end{itemize}
\end{corollary}

\begin{definition}\label{def-4}
	For each $j\in\mathscr{J}$, define the positive and negative eligibility
	index sets by
	$
	I^{+}(j)
	=
	\left\{
	i\in\mathscr{I}:a_{ij}^{+}\geq b_i
	\right\}$,
	$
	I^{-}(j)
	=
	\left\{
	i\in\mathscr{I}:a_{ij}^{-}\geq b_i
	\right\}.
	$
	Thus, $I^{+}(j)$ and $I^{-}(j)$ identify the treatments for which the
	positive and negative clinical relationships of criterion $C_j$,
	respectively, are sufficiently strong to attain the required clinical
	degree.
	
	For each $j\in\mathscr{J}$, define the clinical admissibility interval
	of $C_j$ by
	$
	I_j=\bigcap_{i\in\mathscr{I}}I_{ij}.
	$
	Finally, for each $i\in\mathscr{I}$ and $j\in\mathscr{J}$, define the
	effective evidence activation set by
	$
	S_{ij}^{\prime}=S_{ij}\cap I_j.
	$
	Thus, $S_{ij}^{\prime}$ contains the recommendation levels at which the
	bipolar relationship between $C_j$ and $T_i$ can be activated while
	remaining admissible with respect to all clinical requirements.
\end{definition}

\begin{remark}\label{rmk-1}
	By Corollary~\ref{corl-1} and Definition~\ref{def-4}, the admissible
	recommendation levels of each clinical criterion $C_j$ form an interval
	$
	I_j=[L_j,U_j], ~ \forall j\in\mathscr{J},
	$
	where
	\[
	L_j=
	\begin{cases}
		\displaystyle\max_{i\in I^{-}(j)}\{1-b_i\},
		& I^{-}(j)\neq\varnothing,\\[2mm]
		0,
		& I^{-}(j)=\varnothing,
	\end{cases}
	~
	U_j=
	\begin{cases}
		\displaystyle\min_{i\in I^{+}(j)}\{b_i\},
		& I^{+}(j)\neq\varnothing,\\[2mm]
		1,
		& I^{+}(j)=\varnothing.
	\end{cases}
	\]
	Thus, $L_j$ and $U_j$ represent the lower and upper admissible bounds
	on the recommendation level $x_j$ imposed by all clinical
	requirements involving criterion $C_j$. In particular, the lower bound
	is determined by the strongest relevant negative restriction, whereas
	the upper bound is determined by the strongest relevant positive
	restriction.
	
	Furthermore, for each $i\in\mathscr{I}$ and $j\in\mathscr{J}$ such that
	$S_{ij}^{\prime}\neq\varnothing$, the effective evidence activation set
	$S_{ij}^{\prime}=S_{ij}\cap I_j$ is given by
	\[
	S_{ij}^{\prime}
	=
	\begin{cases}
		\left[
		\max\{L_j,b_i\},
		\min\{U_j,b_i\}
		\right],
		&
		a_{ij}^{+}\geq b_i,~ a_{ij}^{-}<b_i,
		\\[2mm]
		\left[
		\max\{L_j,1-b_i\},
		\min\{U_j,1-b_i\}
		\right],
		&
		a_{ij}^{+}<b_i,~ a_{ij}^{-}\geq b_i,
		\\[2mm]
		\left[
		\max\{L_j,1-b_i\},
		\min\{U_j,1-b_i\}
		\right]
		\cup
		\left[
		\max\{L_j,b_i\},
		\min\{U_j,b_i\}
		\right],
		&
		a_{ij}^{+}\geq b_i,~
		a_{ij}^{-}\geq b_i,~
		b_i>\dfrac{1}{2},
		\\[2mm]
		\left[
		\max\{L_j,1-b_i\},
		\min\{U_j,b_i\}
		\right],
		&
		a_{ij}^{+}\geq b_i,~
		a_{ij}^{-}\geq b_i,~
		b_i \leq \dfrac{1}{2}.
	\end{cases}
	\]
	Hence, $S_{ij}^{\prime}$ identifies the recommendation levels at which
	criterion $C_j$ can activate the bipolar clinical relationship with
	treatment $T_i$ while remaining compatible with the admissible
	requirements imposed by all treatments.
\end{remark}

\begin{example}\label{ex-1}
	Consider the Clinical Decision Support System described by
	Problem~(\ref{eq_1}), with the following fuzzy relational data. The
		positive and negative clinical relationship matrices are given by
	\[
	A^{+}=
	\begin{bmatrix}
		0.80 & 0.30 & 0.60 & 0.40 & 0.50\\
		0.90 & 0.10 & 0.40 & 0.10 & 0.30\\
		0.70 & 0.70 & 0.70 & 0.40 & 0.80\\
		0.80 & 0.70 & 0.30 & 0.50 & 0.50\\
		0.10 & 0.10 & 0.20 & 0.40 & 0.40
	\end{bmatrix},~
	A^{-}=
	\begin{bmatrix}
		0.30 & 0.50 & 0.40 & 0.60 & 0.40\\
		0.10 & 0.30 & 0.90 & 0.30 & 0.10\\
		0.40 & 0.50 & 0.30 & 0.80 & 0.50\\
		0.10 & 0.40 & 0.80 & 0.60 & 0.30\\
		0.80 & 0.30 & 0.40 & 0.20 & 0.70
	\end{bmatrix},~
	b=
	\begin{bmatrix}
		0.60\\
		0.50\\
		0.80\\
		0.70\\
		0.60
	\end{bmatrix}
	\]
	Here, $\mathscr{I}=\{1,2,3,4,5\}$ represents the set of treatment
	alternatives and $\mathscr{J}=\{1,2,3,4,5\}$ represents the set of
	clinical criteria. The $i$-th component $b_i$ specifies the required
	clinical degree associated with treatment alternative $T_i$.
	In this CDSS model, the minimum $t$-norm is employed as the clinical
	aggregation operator. Hence,
	$
	\varphi(u,v)=\min\{u,v\},
	$
	and the $i$-th clinical requirement takes the form
	$
	\max_{j\in\mathscr{J}}
	\left\{
	\max\left\{
	\min(a_{ij}^{+},x_j),
	\min(a_{ij}^{-},1-x_j)
	\right\}
	\right\}
	=b_i.
	$
	To illustrate the construction of the sets introduced above, consider
	the first treatment alternative $T_1$ and the third clinical criterion
	$C_3$. We have
	$
	a_{13}^{+}=0.60,~
	a_{13}^{-}=0.40,~
	b_1=0.60.
	$
	Therefore, the admissibility conditions are
	$
	\min(0.60,x_3)\leq0.60
	$
	and
	$
	\min(0.40,1-x_3)\leq0.60.
	$
	Both inequalities hold for every $x_3\in[0,1]$, and consequently
	$
	I_{13}=[0,1].
	$
	For the corresponding bipolar evidence activation condition, we
	require
	$
	\max\{\min(0.60,x_3),\min(0.40,1-x_3)\}=0.60.
	$
	Since the negative relationship can never attain $0.60$, the equality
	must be attained through the positive relationship. This occurs when
	$
	\min(0.60,x_3)=0.60,
	$
	which gives
	$
	x_3\in[0.60,1].
	$
	Hence,
	$
	S_{13}=[0.60,1].
	$

\end{example}
\section{Characterization of the Clinical Recommendation Solution Set}
\label{sec-3}

The BFRE-constrained clinical decision-support model introduced in
Section \ref{sec-2} identifies the set of recommendation vectors that are
consistent with the prescribed clinical requirements. Before optimizing the
clinical recommendation objective, it is therefore necessary to characterize
the complete feasible region of the BFRE system. This characterization
determines the admissible recommendation levels of the clinical criteria and
provides the mathematical foundation for identifying clinically feasible
recommendation vectors.

For each treatment alternative $T_i$ and clinical criterion $C_j$, the
positive and negative clinical relationships give rise to the corresponding
Positive Evidence Activation Set $S_{ij}^{+}$ and Negative Evidence
Activation Set $S_{ij}^{-}$. Similarly, the associated Positive and Negative
Clinical Admissibility Sets, denoted by $I_{ij}^{+}$ and $I_{ij}^{-}$,
determine the recommendation levels that do not violate the required clinical
degree $b_i$. By combining these sets, we obtain the Bipolar Evidence
Activation Set $S_{ij}$ and the Bipolar Clinical Admissibility Set $I_{ij}$.

The purpose of this section is to use these componentwise sets to construct
the admissible interval for each clinical criterion and, subsequently, to
characterize the Clinical Requirement Satisfaction Set $S_i$ and the complete
Clinical Recommendation Solution Set
$S(A^{+},A^{-},b)$. The resulting characterization provides an explicit
description of all recommendation vectors that satisfy the bipolar clinical
requirements and will be used in the subsequent resolution and optimization
steps.

The following lemma provides two necessary conditions for the feasibility of
Problem~(\ref{eq_1}). The first condition states that the Clinical
Admissibility Interval associated with every clinical criterion must be
nonempty. The second condition ensures that, for each clinical requirement,
at least one clinical criterion can actively contribute to its satisfaction
within the corresponding admissible interval.

\begin{lemma}\label{lm-3}
	Suppose that
	$S\left(A^{+},A^{-},b\right)\neq\varnothing$. Then the following necessary
	conditions hold:
	\begin{itemize}
		\item [\textbf{(a)}]
			The Clinical Admissibility Interval of every clinical criterion
			is nonempty; that is,
			$
			I_j\neq\varnothing,~ \forall j\in\mathscr{J}.
			$
		\item [\textbf{(b)}]
			For every clinical requirement $i\in\mathscr{I}$, there exists
			at least one clinical criterion $j_i\in\mathscr{J}$ such that its Effective
			Evidence Activation Set is nonempty; that is,
			$
			S_{ij_i}^{\prime}\neq\varnothing.
			$
	\end{itemize}
\end{lemma}

\begin{proof}
	\textbf{(a)}
	Suppose, on the contrary, that
	$x\in S\left(A^{+},A^{-},b\right)$ but
	$I_{j_{0}}=\varnothing$ for some $j_{0}\in\mathscr{J}$.
	Since
	$
	I_{j_{0}}=\bigcap_{i\in\mathscr{I}}I_{ij_{0}},
	$
	the emptiness of $I_{j_{0}}$ means that at least one clinical
	requirement cannot be satisfied within the admissible range of the
	recommendation level $x_{j_{0}}$. Thus, there exists
	$i_{0}\in\mathscr{I}$ such that
	$
	x_{j_{0}}\notin I_{i_{0}j_{0}}.
	$
	By Definition~\ref{def-3}, this implies
	$
	\max\left\{
	\min\left(a_{i_{0}j_{0}}^{+},x_{j_{0}}\right),
	\min\left(a_{i_{0}j_{0}}^{-},1-x_{j_{0}}\right)
	\right\}>b_{i_{0}}.
	$
	Hence, the $j_{0}$-th clinical criterion violates the admissibility
	condition associated with the $i_{0}$-th clinical requirement, and
	therefore $x\notin S_{i_{0}}$. This contradicts
	$
	S\left(A^{+},A^{-},b\right)
	=
	\bigcap_{i\in\mathscr{I}}S_i
	$
	and the assumption that
	$x\in S\left(A^{+},A^{-},b\right)$.
	Therefore,
	$
	I_j\neq\varnothing,~ \forall j\in\mathscr{J}.
	$
	\textbf{(b)}
	Suppose, on the contrary, that there exists an
	$i_{0}\in\mathscr{I}$ for which
	$
	S_{i_{0}j}^{\prime}=\varnothing,
	~ \forall j\in\mathscr{J}.
	$
	Since
	$x\in S\left(A^{+},A^{-},b\right)$, we have
	$x\in S_{i_{0}}$. Moreover, part~\textbf{(a)} ensures that
	$I_j\neq\varnothing$ for every $j\in\mathscr{J}$, and feasibility of
	$x$ requires
	$
	x_j\in I_j,~ \forall j\in\mathscr{J}.
	$
	Now, because
	$
	S_{i_{0}j}^{\prime}
	=
	S_{i_{0}j}\cap I_j
	=
	\varnothing,
	~ \forall j\in\mathscr{J},
	$
	and $x_j\in I_j$, we necessarily have
	$
	x_j\notin S_{i_{0}j},
	~ \forall j\in\mathscr{J}.
	$
	Thus, for every clinical criterion $C_j$, the corresponding bipolar
	clinical relationship fails to attain the required clinical degree
	$b_{i_{0}}$. In particular,
	$
	\max\left\{
	\min\left(a_{i_{0}j}^{+},x_j\right),
	\min\left(a_{i_{0}j}^{-},1-x_j\right)
	\right\}<b_{i_{0}},
	~ \forall j\in\mathscr{J}.
	$
	Consequently, no clinical criterion can actively contribute to
	satisfying the $i_{0}$-th clinical requirement at the prescribed degree
	$b_{i_{0}}$. Hence,
	$x\notin S_{i_{0}}$, which contradicts
	$x\in S\left(A^{+},A^{-},b\right)$.
	Therefore, for every $i\in\mathscr{I}$, there exists at least one
	$j_i\in\mathscr{J}$ such that
	$
	S_{ij_i}^{\prime}\neq\varnothing.
	$
\end{proof}

The following lemma provides a necessary and sufficient condition for
determining whether a given clinical recommendation vector
$x\in[0,1]^n$ is feasible for Problem~(\ref{eq_1}). In particular, it
shows that feasibility can be verified through two conditions: each
clinical criterion must have an admissible recommendation level, and
each treatment requirement must have at least one clinical criterion
capable of attaining its prescribed clinical degree.

\begin{lemma}\label{lm-4}
	A vector $x\in[0,1]^n$ belongs to the Clinical Recommendation Solution
	Set $S\left(A^{+},A^{-},b\right)$ if and only if the following
	conditions hold:
	\begin{itemize}
		\item[\textbf{(I)}]
		$
		x_j\in I_j,~ \forall j\in\mathscr{J}.
		$
		That is, the recommendation level assigned to every clinical
		criterion lies within its Clinical Admissibility Interval.
		
		\item[\textbf{(II)}] For every treatment alternative
		$i\in\mathscr{I}$, there exists at least one clinical criterion
		$j_i\in\mathscr{J}$ such that
		$
		x_{j_i}\in S_{ij_i}^{\prime}.
		$
		Thus, for each treatment requirement, at least one clinically
		admissible criterion must actively attain the required clinical
		degree $b_i$.
	\end{itemize}
\end{lemma}

\begin{proof}
	First, suppose that $x\in[0,1]^n$ satisfies conditions
	\textbf{(I)} and \textbf{(II)}.
	From condition~\textbf{(I)}, we have
	$x_j\in I_j$ for every $j\in\mathscr{J}$. Since
	$
	I_j=\bigcap_{i\in\mathscr{I}}I_{ij},
	$
	it follows that
	$x_j\in I_{ij}$ for every $i\in\mathscr{I}$ and
	$j\in\mathscr{J}$. Hence, by Definition~\ref{def-3},
	$
	\max\left\{
	\min\left(a_{ij}^{+},x_j\right),
	\min\left(a_{ij}^{-},1-x_j\right)
	\right\}
	\leq b_i,
	~
	\forall i\in\mathscr{I},\ \forall j\in\mathscr{J}.
	$
	Thus, every clinical criterion operates within the admissible range
	for each treatment requirement.
	On the other hand, condition~\textbf{(II)} states that, for each
	$i\in\mathscr{I}$, there exists at least one clinical criterion
	$j_i\in\mathscr{J}$ such that
	$x_{j_i}\in S_{ij_i}^{\prime}$. Since
	$S_{ij_i}^{\prime}\subseteq S_{ij_i}$, Definition~\ref{def-3} gives
	$
	\max\left\{
	\min\left(a_{ij_i}^{+},x_{j_i}\right),
	\min\left(a_{ij_i}^{-},1-x_{j_i}\right)
	\right\}
	=b_i.
	$
	Therefore, for each treatment requirement $i$, all clinical
	criteria produce values that do not exceed the required degree
	$b_i$, while at least one criterion attains this degree exactly.
	Consequently,
	$
	\max_{j\in\mathscr{J}}
	\left\{
	\max\left[
	\min\left(a_{ij}^{+},x_j\right),
	\min\left(a_{ij}^{-},1-x_j\right)
	\right]
	\right\}
	=b_i,
	~
	\forall i\in\mathscr{I}.
	$
	Hence, $x\in S_i$ for every $i\in\mathscr{I}$. It follows that
	$
	x\in\bigcap_{i\in\mathscr{I}}S_i
	=
	S\left(A^{+},A^{-},b\right).
	$
	Conversely, suppose that
	$x\in S\left(A^{+},A^{-},b\right)$. Then
	$x\in S_i$ for every $i\in\mathscr{I}$. Therefore, for every
	$i\in\mathscr{I}$ and $j\in\mathscr{J}$, the corresponding bipolar
	clinical relationship cannot exceed the required clinical degree,
	which implies
	$
	x_j\in I_{ij},~ \forall i\in\mathscr{I},\
	\forall j\in\mathscr{J}.
	$
	Consequently,
	$
	x_j\in\bigcap_{i\in\mathscr{I}}I_{ij}=I_j,
	~
	\forall j\in\mathscr{J},
	$
	so condition~\textbf{(I)} holds.
	Moreover, since $x\in S_i$ for every $i\in\mathscr{I}$, the defining
	equality for the $i$-th clinical requirement must be attained by at
	least one clinical criterion. Thus, for each $i\in\mathscr{I}$, there
	exists some $j_i\in\mathscr{J}$ such that
	$
	x_{j_i}\in S_{ij_i}.
	$
	Together with $x_{j_i}\in I_{j_i}$, already established by
	condition~\textbf{(I)}, this yields
	$
	x_{j_i}\in S_{ij_i}\cap I_{j_i}
	=S_{ij_i}^{\prime}.
	$
	Hence, condition~\textbf{(II)} also holds.
	Therefore, $x\in S\left(A^{+},A^{-},b\right)$ if and only if
	conditions \textbf{(I)} and \textbf{(II)} are satisfied.
\end{proof}

\begin{corollary}\label{corl-2}
	For each $i\in\mathscr{I}$, a recommendation vector
	$x\in[0,1]^n$ belongs to the Clinical Requirement Satisfaction Set
	$S_i$ if and only if
	$
	x_j\in I_{ij},~ \forall j\in\mathscr{J},
	$
	and there exists at least one $j_i\in\mathscr{J}$ such that
	$
	x_{j_i}\in S_{ij_i}^{\prime}.
	$
	In other words, the $i$-th clinical requirement is satisfied precisely
	when all clinical criteria remain within their corresponding admissible
	ranges and at least one criterion actively attains the required
	clinical degree $b_i$.
\end{corollary}

\begin{proof}
	The result follows directly from Lemma~\ref{lm-4} by considering the
	$i$-th clinical requirement separately. In particular, replacing the
	overall Clinical Recommendation Solution Set
	$S\left(A^{+},A^{-},b\right)$ and the Clinical Admissibility Interval
	$I_j$ in Lemma~\ref{lm-4} with the corresponding requirement-specific
	sets $S_i$ and $I_{ij}$ yields the stated characterization.
\end{proof}

\begin{definition}\label{def-5}
	For each $i\in\mathscr{I}$, define the \emph{Effective Clinical Criterion
		Set} associated with treatment alternative $T_i$ by
	$
	\mathscr{J}_{i}
	=
	\left\{
	j\in\mathscr{J}:S_{ij}^{\prime}\neq\varnothing
	\right\}.
	$
	Thus, $\mathscr{J}_{i}$ contains all clinical criteria whose
	recommendation levels can effectively contribute to satisfying the
	clinical requirement associated with $T_i$ within the corresponding
	Clinical Admissibility Interval.
	Similarly, for each $j\in\mathscr{J}$, define the \emph{Effective
		Treatment Requirement Set} associated with clinical criterion $C_j$ by
	$
	\mathscr{I}_{j}
	=
	\left\{
	i\in\mathscr{I}:S_{ij}^{\prime}\neq\varnothing
	\right\}.
	$
	Hence, $\mathscr{I}_{j}$ identifies all treatment alternatives whose
	clinical requirements can be actively satisfied through criterion
	$C_j$ at an admissible recommendation level.
\end{definition}

\begin{definition}\label{def-6}
	A function
	$e:\mathscr{I}\rightarrow\mathscr{J}$ is called an
	\emph{Admissible Clinical Evidence-Assignment Function} if
	$
	e(i)\in\mathscr{J}_{i}(e),~ \forall i\in\mathscr{I},
	$
	where $\mathscr{J}_{i}(e)$ is determined recursively as follows.
	
	\begin{itemize}
		\item[\textbf{(I)}] For the first treatment alternative,
		$
		\mathscr{J}_{1}(e)=\mathscr{J}_{1}.
		$
		Thus, any clinical criterion that can effectively satisfy the
		first clinical requirement may be assigned to $T_1$.
		
		\item[\textbf{(II)}] For each $j\in\mathscr{J}$ and
		$i\in\mathscr{I}\setminus\{1\}$, define
		$
		\mathscr{I}_{j}(e,i)
		=
		\left\{
		k\in\mathscr{I}:1\leq k<i,\ e(k)=j
		\right\}.
		$
		Hence, $\mathscr{I}_{j}(e,i)$ contains the treatment indices
		preceding $i$ that have already been assigned the same clinical
		criterion $C_j$.
		
		\item[\textbf{(III)}] For each
		$i\in\mathscr{I}\setminus\{1\}$, define
		$
		\mathscr{J}_{i}(e)
		=
		\left\{
		j\in\mathscr{J}_{i}:
		\mathscr{I}_{j}(e,i)=\varnothing
		\ \text{or}\
		S_{ij}^{\prime}
		\cap
		\left(
		\bigcap_{k\in\mathscr{I}_{j}(e,i)}
		S_{kj}^{\prime}
		\right)
		\neq\varnothing
		\right\}.
		$
		Thus, a clinical criterion $C_j$ remains available for assignment
		to $T_i$ either when it has not previously been assigned to any
		treatment alternative, or when the effective evidence activation
		sets associated with all treatment alternatives already assigned
		to $C_j$ have a common recommendation level with
		$S_{ij}^{\prime}$.
	\end{itemize}
	
	Let $E$ denote the set of all Admissible Clinical Evidence-Assignment
	Functions. For convenience, each $e\in E$ can be represented by the
	vector
	$
	e=[j_1,\ldots,j_m],
	$
	where
	$
	e(i)=j_i,~ \forall i\in\mathscr{I}.
	$
	Thus, $j_i$ identifies the clinical criterion selected as an effective
	evidence contributor for the treatment alternative $T_i$.
\end{definition}

\begin{definition}\label{def-7}
	For each admissible clinical evidence-assignment function
	$e\in E$ and each $j\in\mathscr{J}$, define
	$
	\mathscr{I}_{j}(e)
	=
	\left\{
	i\in\mathscr{I}:e(i)=j
	\right\}.
	$
	Thus, $\mathscr{I}_{j}(e)$ contains the treatment alternatives whose
	corresponding clinical requirements select criterion $C_j$ as an
	effective evidence contributor under the assignment $e$.
	
	For each $e\in E$, let $S(e)$ denote the set of all clinical
	recommendation vectors
	$x=(x_1,\ldots,x_n)\in[0,1]^n$ satisfying
	\begin{equation}\label{eq-5}
		x_j\in
		\begin{cases}
			\displaystyle
			\bigcap_{i\in\mathscr{I}_{j}(e)} S_{ij}^{\prime},
			& \mathscr{I}_{j}(e)\neq\varnothing,\\[2ex]
			I_j,
			& \mathscr{I}_{j}(e)=\varnothing,
		\end{cases}
		~ \forall j\in\mathscr{J}.
	\end{equation}
	Accordingly, define the \emph{Admissible Recommendation-Level Set}
	associated with criterion $C_j$ under $e$ by
	\[
	S(e)_j=
	\begin{cases}
		\displaystyle
		\bigcap_{i\in\mathscr{I}_{j}(e)}S_{ij}^{\prime},
		& \mathscr{I}_{j}(e)\neq\varnothing,\\[2ex]
		I_j,
		& \mathscr{I}_{j}(e)=\varnothing.
	\end{cases}
	\]
	Then
	$
	S(e)=S(e)_1\times\cdots\times S(e)_n.
	$
	In this way, $S(e)$ represents the \emph{Clinical Recommendation
		Region Associated with the Admissible Evidence Assignment} $e$. For a
	criterion $C_j$ that has been selected by one or more treatment
	requirements, its recommendation level must belong to the intersection
	of the corresponding Effective Evidence Activation Sets. If no
	treatment requirement selects $C_j$ under $e$, its recommendation level
	may take any value in the Clinical Admissibility Interval $I_j$.
\end{definition}

The recursive condition in Definition~\ref{def-6} can be expressed in a simpler global form. The following corollary shows that the admissibility of an evidence-assignment function is completely determined by the compatibility of the Effective Evidence Activation
Sets corresponding to each clinical criterion. More precisely, whenever several treatment requirements are assigned to the same clinical criterion, their corresponding Effective Evidence Activation Sets must have a nonempty common intersection. This result provides a convenient criterion for identifying admissible functions and will be used in the construction of the feasible clinical recommendation set.

\begin{corollary}\label{corl-3}
	Let
	$
	e:\mathscr{I}\rightarrow\bigcup_{i\in\mathscr{I}}\mathscr{J}_i
	$
	be a function satisfying
	$
	e(i)\in\mathscr{J}_i,~ \forall i\in\mathscr{I}.
	$
	Then $e$ is an Admissible Clinical Evidence-Assignment Function,
	that is, $e\in E$, if and only if
	$
	\bigcap_{i\in\mathscr{I}_j(e)}S_{ij}^{\prime}\neq\varnothing
	$
	for every $j\in\mathscr{J}$ such that
	$\mathscr{I}_j(e)\neq\varnothing$.
	In other words, an assignment is admissible precisely when, for every
	clinical criterion selected by one or more treatment requirements,
	there exists at least one common recommendation level at which that
	criterion can simultaneously satisfy all of the corresponding
	requirements.
\end{corollary}

\begin{proof}
	Suppose first that $e\in E$. Let $j_0\in\mathscr{J}$ be such that
	$\mathscr{I}_{j_0}(e)\neq\varnothing$, and let
	$
	i_0=\max\mathscr{I}_{j_0}(e).
	$
	By the definition of $\mathscr{I}_{j_0}(e,i_0)$, we have
	$
	\mathscr{I}_{j_0}(e)
	=
	\mathscr{I}_{j_0}(e,i_0)\cup\{i_0\}.
	$
	Since $e$ is admissible, Definition~\ref{def-6} implies
	$
	S_{i_0j_0}^{\prime}
	\cap
	\left(
	\bigcap_{k\in\mathscr{I}_{j_0}(e,i_0)}
	S_{kj_0}^{\prime}
	\right)
	\neq\varnothing.
	$
	Therefore,
	$
	\bigcap_{i\in\mathscr{I}_{j_0}(e)}
	S_{ij_0}^{\prime}
	\neq\varnothing.
	$
	Since $j_0$ was arbitrary, the required intersection is nonempty for
	every $j\in\mathscr{J}$ with
	$\mathscr{I}_j(e)\neq\varnothing$.
	Conversely, suppose that
	$
	\bigcap_{i\in\mathscr{I}_j(e)}
	S_{ij}^{\prime}\neq\varnothing
	$
	for every $j\in\mathscr{J}$ such that
	$\mathscr{I}_j(e)\neq\varnothing$.
	Because $e(i)\in\mathscr{J}_i$ for every $i\in\mathscr{I}$, it remains
	only to verify the compatibility condition in
	Definition~\ref{def-6}.
	Consider any $i_0\in\mathscr{I}\setminus\{1\}$ and let
	$j_0=e(i_0)$. If
	$\mathscr{I}_{j_0}(e,i_0)=\varnothing$, the required condition holds
	automatically. Otherwise,
	$
	\mathscr{I}_{j_0}(e,i_0)\subseteq\mathscr{I}_{j_0}(e),
	$
	and, since $i_0\in\mathscr{I}_{j_0}(e)$, the assumption gives
	$
	\bigcap_{i\in\mathscr{I}_{j_0}(e)}
	S_{ij_0}^{\prime}\neq\varnothing.
	$
	Hence there exists a recommendation level belonging simultaneously to
	$S_{i_0j_0}^{\prime}$ and to every
	$S_{kj_0}^{\prime}$ with
	$k\in\mathscr{I}_{j_0}(e,i_0)$. Thus,
	$
	S_{i_0j_0}^{\prime}
	\cap
	\left(
	\bigcap_{k\in\mathscr{I}_{j_0}(e,i_0)}
	S_{kj_0}^{\prime}
	\right)
	\neq\varnothing.
	$
	Therefore, the compatibility condition in Definition~\ref{def-6} is
	satisfied for every $i\in\mathscr{I}$, and consequently $e\in E$.
\end{proof}

\begin{remark}\label{rmk-2}
	Corollary~\ref{corl-3} also provides an immediate upper bound on the
	number of admissible evidence-assignment functions. Indeed, for each
	$i\in\mathscr{I}$, the value $e(i)$ can be selected from
	$\mathscr{J}_{i}$, and hence
	$
	|E|
	\leq
	\prod_{i\in\mathscr{I}}|\mathscr{J}_{i}|,
	$
	where $|\mathscr{J}_{i}|$ denotes the cardinality of
	$\mathscr{J}_{i}$. This bound corresponds to the number of all possible
	criterion assignments before the compatibility conditions in
	Corollary~\ref{corl-3} are imposed. In general, many of these
	assignments are excluded because the corresponding Effective Evidence
	Activation Sets do not have a common intersection. Therefore, the
	actual number of admissible functions is typically considerably smaller
	than the above upper bound.
\end{remark}

The preceding results provide the necessary ingredients for a complete
characterization of the feasible clinical recommendation set. In
particular, each admissible evidence-assignment function $e\in E$
identifies a collection of compatible recommendation levels for the
clinical criteria. The following theorem shows that every feasible
clinical recommendation belongs to one of these regions, and,
conversely, every vector contained in such a region is feasible.
Consequently, the complete feasible solution set can be obtained as
the union of the recommendation regions associated with all admissible
evidence assignments.

\begin{theorem}\label{thm-1}
	The feasible solution set of Problem~(\ref{eq_1}) is given by
	$
	S\left(A^{+},A^{-},b\right)
	=
	\bigcup_{e\in E}S(e).
	$
\end{theorem}

\begin{proof}
	We first prove that
	$
	\bigcup_{e\in E}S(e)
	\subseteq
	S\left(A^{+},A^{-},b\right).
	$
	Let
	$
	x\in\bigcup_{e\in E}S(e).
	$
	Then there exists an admissible function $e_0\in E$ such that
	$x\in S(e_0)$. According to~(\ref{eq-5}), for each
	$j\in\mathscr{J}$, either
	$
	x_j\in
	\bigcap_{i\in\mathscr{I}_j(e_0)}S_{ij}^{\prime}
	~\text{if }\mathscr{I}_j(e_0)\neq\varnothing,
	$
	or
	$
	x_j\in I_j
	~\text{if }\mathscr{I}_j(e_0)=\varnothing.
	$
	In the first case, since
	$S_{ij}^{\prime}=S_{ij}\cap I_j$ by Definition~\ref{def-4},
	we also have $x_j\in I_j$. Hence,
	$
	x_j\in I_j,~ \forall j\in\mathscr{J}.
	$
	Moreover, since $e_0\in E$, Definition~\ref{def-6} gives
	$
	e_0(i)=j_i\in\mathscr{J}_i(e_0)\subseteq\mathscr{J}_i,
	~ \forall i\in\mathscr{I}.
	$
	Thus,
	$
	\mathscr{I}_{j_i}(e_0)\neq\varnothing
	$
	for every $i\in\mathscr{I}$. Consequently, by~(\ref{eq-5}),
	$
	x_{j_i}
	\in
	\bigcap_{k\in\mathscr{I}_{j_i}(e_0)}
	S_{kj_i}^{\prime}
	\subseteq
	S_{ij_i}^{\prime},
	~ \forall i\in\mathscr{I}.
	$
	Therefore, for every treatment requirement $i$, there exists a
	criterion $j_i$ such that
	$
	x_{j_i}\in S_{ij_i}^{\prime}.
	$
	Together with $x_j\in I_j$ for all $j\in\mathscr{J}$, Lemma~\ref{lm-4}
	implies
	$
	x\in S\left(A^{+},A^{-},b\right).
	$
	Hence,
	$
	\bigcup_{e\in E}S(e)
	\subseteq
	S\left(A^{+},A^{-},b\right).
	$
	For the converse inclusion, let
	$
	x\in S\left(A^{+},A^{-},b\right).
	$
	Define
	$
	\mathscr{I}_j(x)
	=
	\left\{
	i\in\mathscr{I}:x_j\in S_{ij}^{\prime}
	\right\},
	~ j\in\mathscr{J},
	$
	and
	$
	\mathscr{J}_i(x)
	=
	\left\{
	j\in\mathscr{J}:x_j\in S_{ij}^{\prime}
	\right\},
	~ i\in\mathscr{I}.
	$
	By Lemma~\ref{lm-4}, for every $i\in\mathscr{I}$ there exists at
	least one $j\in\mathscr{J}$ such that
	$x_j\in S_{ij}^{\prime}$. Hence,
	$
	\mathscr{J}_i(x)\neq\varnothing,
	~ \forall i\in\mathscr{I}.
	$
	Furthermore, if $j\in\mathscr{J}_i(x)$, then
	$S_{ij}^{\prime}\neq\varnothing$, and therefore
	$j\in\mathscr{J}_i$ by Definition~\ref{def-5}.
	For each $i\in\mathscr{I}$, choose
	$
	j_i=\min\mathscr{J}_i(x)
	$
	and define $e_0:\mathscr{I}\to\bigcup_{i\in\mathscr{I}}\mathscr{J}_i$
	by
	$
	e_0(i)=j_i.
	$
	Then
	\begin{equation}\label{eq-6}
		e_0(i)=j_i\in\mathscr{J}_i,
		~ \forall i\in\mathscr{I}.
	\end{equation}
	Consider a criterion $C_j$ that is selected by at least one treatment
	requirement under $e_0$, that is,
	$\mathscr{I}_j(e_0)\neq\varnothing$. If
	$k\in\mathscr{I}_j(e_0)$, then $e_0(k)=j$. By the construction of
	$e_0$,
	$
	j=\min\mathscr{J}_k(x),
	$
	and hence
	$
	x_j\in S_{kj}^{\prime}.
	$
	Therefore,
	$
	x_j\in
	\bigcap_{k\in\mathscr{I}_j(e_0)}
	S_{kj}^{\prime}
	$
	for every $j$ with
	$\mathscr{I}_j(e_0)\neq\varnothing$. Thus,
	\begin{equation}\label{eq-7}
		\bigcap_{i\in\mathscr{I}_j(e_0)}
		S_{ij}^{\prime}\neq\varnothing,
		~
		\forall j\in\mathscr{J}
		\text{ such that }
		\mathscr{I}_j(e_0)\neq\varnothing.
	\end{equation}
	By Corollary~\ref{corl-3}, (\ref{eq-6}) and (\ref{eq-7}) imply that
	$
	e_0\in E.
	$
	Finally, Lemma~\ref{lm-4} gives
	$
	x_j\in I_j,~ \forall j\in\mathscr{J}.
	$
	Hence, if $\mathscr{I}_j(e_0)=\varnothing$, then
	$x_j\in I_j$, while if $\mathscr{I}_j(e_0)\neq\varnothing$, then
	$
	x_j\in
	\bigcap_{i\in\mathscr{I}_j(e_0)}
	S_{ij}^{\prime}.
	$
	According to~(\ref{eq-5}), these conditions imply
	$
	x\in S(e_0).
	$
	Since $e_0\in E$, it follows that
	$
	x\in\bigcup_{e\in E}S(e).
	$
	Therefore,
	$
	S\left(A^{+},A^{-},b\right)
	\subseteq
	\bigcup_{e\in E}S(e).
	$
	Combining the two inclusions yields
	$
	S\left(A^{+},A^{-},b\right)
	=
	\bigcup_{e\in E}S(e).
	$
\end{proof}

The preceding results allow us to establish an important structural
property of the feasible clinical recommendation set. For each
admissible evidence-assignment function $e\in E$, the set $S(e)$ is
constructed coordinatewise from either an Effective Evidence
Activation Set or a Clinical Admissibility Interval. Since these sets
are closed and bounded, each $S(e)$ is compact. Consequently, the
following result shows that the feasible clinical recommendation set
is represented as a finite union of compact recommendation regions.

Suppose that $e\in E$, $x\in S(e)$, and $j\in\mathscr{J}$. According to
(\ref{eq-5}), either
$
x_j\in
\bigcap_{i\in\mathscr{I}_j(e)}S_{ij}^{\prime},
~
\mathscr{I}_j(e)\neq\varnothing,
$
or
$
x_j\in I_j=[L_j,U_j],
~
\mathscr{I}_j(e)=\varnothing.
$
In the former case, Corollary~\ref{corl-3} guarantees that
$
\bigcap_{i\in\mathscr{I}_j(e)}S_{ij}^{\prime}\neq\varnothing.
$
Moreover, each $S_{ij}^{\prime}$ is a closed and bounded subset of
$[0,1]$. Hence, their finite intersection is also closed and bounded.
Although this intersection need not be connected, it is compact.
The latter case is immediate because $I_j=[L_j,U_j]$ is a closed and
bounded interval. Therefore, every coordinate set $S(e)_j$ is compact,
and hence
$
S(e)=S(e)_1\times\cdots\times S(e)_n
$
is a compact subset of $[0,1]^n$.
Since $E$ is finite, Theorem~\ref{thm-1} consequently implies that
the feasible clinical recommendation set
$S(A^{+},A^{-},b)$ is a finite union of compact sets. These compact
sets are not necessarily connected, reflecting the potentially
disconnected structure of the feasible clinical recommendation region.

\begin{example}\label{ex-2}
	Consider the CDSS problem stated in Example~\ref{ex-1}. From
	Definition~\ref{def-5} and the Effective Evidence Activation Sets
	obtained in Example~\ref{ex-1}, we have
	$\mathscr{J}_{1}=\{3,4\}$,
	$\mathscr{J}_{2}=\{1,3\}$,
	$\mathscr{J}_{3}=\{4,5\}$,
	$\mathscr{J}_{4}=\{2\}$,
	$\mathscr{J}_{5}=\{1,5\}.$
	Hence, according to Remark~\ref{rmk-2}, the number of possible
	evidence-assignment functions is bounded above by
	$
	\prod_{i\in\mathscr{I}}|\mathscr{J}_{i}|
	=
	2\times2\times2\times1\times2
	=16.
	$
	This is only an upper bound, since not all such assignments satisfy the
	compatibility condition required for admissibility.
	To illustrate this point, consider first the function
	$
	e_{1}=[3,3,4,2,1].
	$
	Clearly, $e_1(i)\in\mathscr{J}_i$ for every $i\in\mathscr{I}$. However,
	$e_1(1)=e_1(2)=3$, and therefore
	$
	\mathscr{I}_{3}(e_1)=\{1,2\}.
	$
	From Example~\ref{ex-1},
	$
	S_{13}^{\prime}=[0.6,1],
	~
	S_{23}^{\prime}=\{0.5\}.
	$
	Consequently,
	$
	\bigcap_{i\in\mathscr{I}_{3}(e_1)}S_{i3}^{\prime}
	=
	S_{13}^{\prime}\cap S_{23}^{\prime}
	=
	\varnothing.
	$
	Thus, by Corollary~\ref{corl-3}, $e_1\notin E$.
	In the clinical interpretation, criterion $C_3$ cannot simultaneously
	serve as effective evidence for $T_1$ and $T_2$, because there is no
	common recommendation level of $C_3$ satisfying both corresponding
	requirements.
	Now consider
	$
	e_{2}=[3,1,4,2,5].
	$
	Again, $e_2(i)\in\mathscr{J}_i$ for every $i\in\mathscr{I}$. Moreover,
	each selected clinical criterion has a nonempty effective activation
	set compatible with the corresponding treatment requirement. In
	particular,
	$\mathscr{I}_{1}(e_2)=\{2\}$,
	$\mathscr{I}_{2}(e_2)=\{4\}$,
	$\mathscr{I}_{3}(e_2)=\{1\}$,
	$\mathscr{I}_{4}(e_2)=\{3\}$,
	$\mathscr{I}_{5}(e_2)=\{5\}.$
	Hence, all intersections in Corollary~\ref{corl-3} are nonempty, and
	therefore
	$
	e_2\in E.
	$
	For this admissible evidence-assignment function, Definition~\ref{def-7}
	gives
	$S(e_2)_1=S_{21}^{\prime}=\{0.5\}$,
	$S(e_2)_2=S_{42}^{\prime}=[0.7,1]$,
	$S(e_2)_3=S_{13}^{\prime}=[0.6,1]$,
	$S(e_2)_4=S_{34}^{\prime}=[0,0.2]$,
	and
	$
	S(e_2)_5=S_{55}^{\prime}=\{0.4\}.
	$
	Therefore, the clinical recommendation region associated with $e_2$
	is
	$
	S(e_2)
	=
	\{0.5\}\times[0.7,1]\times[0.6,1]
	\times[0,0.2]\times\{0.4\}.
	$
	Thus, $S(e_2)$ contains all clinical recommendation vectors that are
	compatible with the particular admissible evidence-assignment pattern
	$e_2$. By Theorem~\ref{thm-1}, the complete feasible clinical
	recommendation set is obtained by taking the union of such regions over
	all admissible functions $e\in E$.
\end{example}

The preceding results provide the mathematical basis for generating
clinically admissible recommendations and identifying an optimal
recommendation within the proposed BFRE-based Clinical Decision Support
System. The resulting procedure is summarized in Algorithm 1.
It first determines the admissible recommendation levels and verifies
the feasibility of the clinical BFRE system. After applying the
simplification techniques developed for the general BFRE framework
\cite{ref_63}, it constructs the feasible clinical recommendation set
and subsequently determines the local and global clinical optimization
stages through the evidence-assignment-based optimization
procedure.

\begin{algorithm}[!ht]
	\label{myalgorithm1}
	\DontPrintSemicolon
	\SetKwInOut{Input}{Input}
	\SetKwInOut{Output}{Output}
	
	\Input{Fuzzy relational matrices $A^{+}$ and $A^{-}$, fuzzy requirement vector $b$, and objective function $f$, defining Problem~\eqref{eq_1}.}
	\Output{The feasible clinical recommendation set
		$S(A^{+},A^{-},b)$ and a globally optimal clinical recommendation.}
	
	Compute $I_{ij}$, $S_{ij}$, $I_j$, and $S_{ij}^{\prime}$ for all
	$i\in\mathscr{I}$ and $j\in\mathscr{J}$ to determine the clinical
	admissibility intervals and effective evidence activation sets
	(Corollary~\ref{corl-1}, Definition~\ref{def-4}).\;
	
	\If{$I_j=\varnothing$ for some $j\in\mathscr{J}$}{
		terminate; the corresponding clinical criterion has no admissible
		recommendation level, and the CDSS problem is infeasible
		(Lemma~\ref{lm-3}(a)).\;
	}
	\Else{
		\If{$S_{ij}^{\prime}=\varnothing$ for every $j\in\mathscr{J}$ for
			some $i\in\mathscr{I}$}{
			terminate; no clinical criterion can effectively satisfy the
			requirement associated with treatment $T_i$, and the CDSS
			problem is infeasible (Lemma~\ref{lm-3}(b)).\;
		}
		\Else{
			Apply the simplification rules developed for the general BFRE framework
			\cite{ref_63} to reduce the clinical BFRE system by eliminating
			redundant equations and fixing recommendation levels when possible.\;
			
			For each admissible Clinical Evidence-Assignment Function
			$e\in E$, construct the corresponding clinical recommendation
			region $S(e)$ according to Definitions~\ref{def-6} and
			\ref{def-7}, and obtain
			$
			S(A^{+},A^{-},b)
			=
			\bigcup_{e\in E}S(e).
			$
			
			For each $e\in E$, construct the candidate clinical
			recommendation $x^{*}(e)$ according to Definition~\ref{def-8}
			and evaluate $f(x^{*}(e))$.\;
			
			Select
			$
			e^{*}\in
			\operatorname*{arg\,min}_{e\in E}
			f\left(x^{*}(e)\right).
			$
			Then, by Theorems~\ref{thm-2} and~\ref{thm-3},
			$x^{*}(e^{*})$ is a globally optimal clinical recommendation,
			with
			$
			f\left(x^{*}(e^{*})\right)
			=
			\min_{e\in E}f\left(x^{*}(e)\right).
			$
		}
	}
	\caption{BFRE-Based Clinical Recommendation Generation and
		Optimization Algorithm}
\end{algorithm}

\section{Local and Global Optimal Clinical Recommendations}
\label{sec-5}
The fuzzy relational inputs determine feasibility, whereas the objective
function determines preference within the feasible set. Its coefficients or
parameters may represent costs, burdens, risks, or decision-maker preferences.
The mathematical results below require the stated continuity and monotonicity
properties.

Throughout this section, we assume that the clinical objective function
$f:\mathbb{R}^{n}\rightarrow\mathbb{R}$ is continuous and that its
monotonicity with respect to each clinical recommendation level $x_j$
is determined by the two index sets $\mathscr{J}^{+}$ and
$\mathscr{J}^{-}$. In particular, $f$ is non-decreasing in $x_j$ for
every $j\in\mathscr{J}^{+}$ and non-increasing in $x_j$ for every
$j\in\mathscr{J}^{-}$. Thus, the sets $\mathscr{J}^{+}$ and
$\mathscr{J}^{-}$ classify the clinical criteria according to the
direction in which their recommendation levels affect the optimization
objective. For a criterion $C_j$ with $j\in\mathscr{J}^{+}$, increasing
$x_j$ cannot improve the clinical objective, whereas for a criterion
$C_j$ with $j\in\mathscr{J}^{-}$, increasing $x_j$ cannot worsen the
clinical objective.

\begin{definition}\label{def-8}
Suppose that
$
S\left(A^{+},A^{-},b\right)\neq\varnothing.
$
For each admissible Clinical Evidence-Assignment Function $e\in E$, we
define a candidate optimal clinical recommendation vector
$
x^{*}(e)
=
\left(x^{*}(e)_{1},\ldots,x^{*}(e)_{n}\right).
$
The component $x^{*}(e)_j$ represents the recommendation level assigned
to clinical criterion $C_j$ under the evidence-assignment pattern $e$.
For each $j\in\mathscr{J}$, this component is selected from the
corresponding admissible recommendation region according to the
monotonicity of the clinical objective:
\begin{equation}\label{eq-9}
	x^{*}(e)_{j}= \begin{cases}\min \left\{\bigcap_{i \in \mathscr{I}_{j}(e)} S_{i j}^{\prime}\right\} & , j \in \mathscr{J}^{+} ~\text {and } \mathscr{I}_{j}(e) \neq \varnothing  \\ L_{j} & , j \in \mathscr{J}^{+} ~\text {and } \mathscr{I}_{j}(e)=\varnothing \\ \max \left\{\bigcap_{i \in \mathscr{I}_{j}(e)} S_{i j}^{\prime}\right\} & , j \in \mathscr{J}^{-} ~\text {and } \mathscr{I}_{j}(e) \neq \varnothing \\ U_{j} & , j \in \mathscr{J}^{-} ~\text {and } \mathscr{I}_{j}(e)=\varnothing\end{cases}
\end{equation}
Here,
$
I_j=[L_j,U_j]
$
is the Clinical Admissibility Interval associated with clinical
criterion $C_j$, while $S_{ij}^{\prime}$ denotes the Effective Evidence
Activation Set associated with treatment alternative $T_i$ and
clinical criterion $C_j$.

If $\mathscr{I}_{j}(e)\neq\varnothing$, then clinical criterion $C_j$
is selected by at least one treatment requirement under the assignment
$e$. Consequently, its recommendation level must belong to the common
intersection
$
\bigcap_{i\in\mathscr{I}_{j}(e)}S_{ij}^{\prime},
$
which contains the recommendation levels at which $C_j$ can effectively
contribute to all corresponding treatment requirements simultaneously.
For $j\in\mathscr{J}^{+}$, the minimum element of this intersection is
selected because increasing $x_j$ cannot improve the clinical
objective. Conversely, for $j\in\mathscr{J}^{-}$, the maximum element is
selected because increasing $x_j$ cannot worsen the clinical objective.

If $\mathscr{I}_{j}(e)=\varnothing$, then criterion $C_j$ is not selected
as an effective evidence contributor by any treatment requirement under
$e$. Hence, its recommendation level is chosen directly from its
Clinical Admissibility Interval $I_j=[L_j,U_j]$. In this case, $L_j$ is
selected for $j\in\mathscr{J}^{+}$, whereas $U_j$ is selected for
$j\in\mathscr{J}^{-}$, consistently with the monotonicity of the
clinical objective.

Finally, define the set of candidate optimal clinical recommendation
vectors by
$
F^{*}
=
\left\{
x^{*}(e):e\in E
\right\}.
$
Thus, $F^{*}$ contains one objective-preferred clinical recommendation
vector for each admissible Clinical Evidence-Assignment Function and
provides the collection of candidate solutions to be considered in the
subsequent determination of local and global clinical optimization
stages.
\end{definition}

The following theorem establishes the role of the vectors
$x^{*}(e)$ in the optimization stage of the proposed BFRE-based
Clinical Decision Support System. For each admissible Clinical
Evidence-Assignment Function $e\in E$, the vector $x^{*}(e)$ represents
a clinical recommendation that is feasible for the corresponding
clinical recommendation region $S(e)$ and is optimal within that region.
Thus, the collection $F^{*}$ provides a set of candidate clinical
recommendations from which the overall optimal recommendation can be
identified.

\begin{theorem}
	\label{thm-2}
	Suppose that
	$
	S\left(A^{+},A^{-},b\right)\neq\varnothing.
	$
	Then the following statements hold:
	\begin{itemize}
		\item[\textbf{(a)}]
		Every candidate clinical recommendation generated by an admissible
		Clinical Evidence-Assignment Function is feasible for the complete
		BFRE-based CDSS model; that is,
		$
		F^{*}\subseteq S\left(A^{+},A^{-},b\right).
		$
		
		\item[\textbf{(b)}]
		For every $e\in E$, the recommendation vector $x^{*}(e)$ is a global
		optimum of the clinical objective function over the corresponding
		clinical recommendation region $S(e)$.
	\end{itemize}
\end{theorem}

\begin{proof}
	\textbf{(a)}
	Let
	$
	x^{*}(e_{0})\in F^{*}
	$
	for some $e_{0}\in E$. By the definition of $x^{*}(e_{0})$ in
	\eqref{eq-9}, together with the characterization of the corresponding
	clinical recommendation region in \eqref{eq-5}, we have
	$
	x^{*}(e_{0})\in S(e_{0}).
	$
	Since
	$
	S(e_{0})\subseteq\bigcup_{e\in E}S(e),
	$
	it follows that
	$
	x^{*}(e_{0})
	\in
	\bigcup_{e\in E}S(e).
	$
	By Theorem~1, the union of all clinical recommendation regions is
	contained in the complete BFRE solution set. Hence,
	$
	x^{*}(e_{0})
	\in
	S\left(A^{+},A^{-},b\right).
	$
	Since $x^{*}(e_{0})$ was chosen arbitrarily from $F^{*}$, we conclude
	that
	$
	F^{*}\subseteq S\left(A^{+},A^{-},b\right).
	$
	Therefore, every recommendation generated by an admissible Clinical
	Evidence-Assignment Function is a feasible clinical recommendation for
	the complete CDSS model.
	\textbf{(b)}
	Let
	$
	x\in S(e)
	$
	be any clinically admissible recommendation associated with the same
	Clinical Evidence-Assignment Function $e$. By the definition of
	$S(e)$ in \eqref{eq-5} and the construction of $x^{*}(e)$ in
	\eqref{eq-9}, we have
	$
	x^{*}(e)_{j}\leq x_{j},
	~
	\forall j\in\mathscr{J}^{+},
	$
	and
	$
	x_{j}\leq x^{*}(e)_{j},
	~
	\forall j\in\mathscr{J}^{-}.
	$
	For criteria $C_j$ with $j\in\mathscr{J}^{+}$, increasing the
	recommendation level cannot improve the clinical objective because
	$f$ is non-decreasing in $x_j$. Hence,
	$
	x^{*}(e)_{j}\leq x_j
	~\Longrightarrow~
	f\left(x^{*}(e)\right)
	\leq f(x)
	$
	with respect to these components.
	Similarly, for criteria $C_j$ with $j\in\mathscr{J}^{-}$, increasing
	the recommendation level cannot worsen the clinical objective because
	$f$ is non-increasing in $x_j$. Therefore,
	$
	x_j\leq x^{*}(e)_{j}
	~\Longrightarrow~
	f\left(x^{*}(e)\right)
	\leq f(x)
	$
	with respect to these components.
	Combining these inequalities for all $j\in\mathscr{J}$ gives
	$
	f\left(x^{*}(e)\right)\leq f(x),
	~
	\forall x\in S(e).
	$
	Consequently,
	$
	x^{*}(e)
	=
	\underset{x\in S(e)}{\operatorname{arg\,min}}\,
	f(x),
	$
	and therefore $x^{*}(e)$ is a global optimum of the clinical objective
	over the clinical recommendation region $S(e)$.
\end{proof}

The preceding results establish that, for each admissible Clinical
Evidence-Assignment Function $e\in E$, the vector $x^{*}(e)$ is a
global optimum of the clinical objective over the corresponding
clinical recommendation region $S(e)$. The next result connects these
region-wise optimal clinical recommendations to the global optimization
problem defined over the complete BFRE-based Clinical Decision Support
System.

In the CDSS framework, let $S^{*}$ denote the \emph{set of globally
	optimal clinical recommendation vectors}, that is,
\[
S^{*}
=
\left\{
x\in S\left(A^{+},A^{-},b\right):
f(x)
=
\min_{y\in S(A^{+},A^{-},b)}f(y)
\right\}.
\]
Thus, $S^{*}$ contains all clinically admissible recommendation
vectors that attain the best possible value of the specified clinical
objective over the complete BFRE solution set. If several clinical
recommendations provide the same optimal objective value, all of them
belong to $S^{*}$. In this sense, $S^{*}$ represents the collection of
globally preferred clinical recommendations produced by the proposed
CDSS optimization framework.

\begin{theorem}	\label{thm-3}
	Suppose that
	$
	S\left(A^{+},A^{-},b\right)\neq\varnothing
	$
	and let $S^{*}$ denote the set of globally optimal clinical
	recommendation vectors defined by
	$
	S^{*}
	=
	\left\{
	x\in S\left(A^{+},A^{-},b\right):
	f(x)
	=
	\min_{y\in S(A^{+},A^{-},b)}f(y)
	\right\}.
	$
	If there exists $e^{*}\in E$ such that
	$
	f\left(x^{*}(e^{*})\right)
	=
	\min_{e\in E}
	f\left(x^{*}(e)\right),
	$
	then
	$
	x^{*}(e^{*})\in S^{*}.
	$
	In other words, the candidate clinical recommendation having the
	smallest clinical objective value among all evidence-assignment-based
	candidate recommendations is a globally optimal clinical
	recommendation for Problem~\eqref{eq_1}.
\end{theorem}

\begin{proof}
	Let
	$
	x\in S\left(A^{+},A^{-},b\right)
	$
	be an arbitrary feasible clinical recommendation vector. By
	Theorem~\ref{thm-1}, there exists an admissible Clinical
	Evidence-Assignment Function $e_{0}\in E$ such that
	$
	x\in S(e_{0}).
	$
	By Theorem~\ref{thm-2}, $x^{*}(e_{0})$ is a global optimum of the
	clinical objective over $S(e_{0})$. Therefore,
	$
	f\left(x^{*}(e_{0})\right)\leq f(x).
	$
	Moreover, by the definition of $e^{*}$,
	$
	f\left(x^{*}(e^{*})\right)
	\leq
	f\left(x^{*}(e_{0})\right).
	$
	Combining these two inequalities gives
	$
	f\left(x^{*}(e^{*})\right)
	\leq
	f\left(x^{*}(e_{0})\right)
	\leq
	f(x).
	$
	Hence,
	$
	f\left(x^{*}(e^{*})\right)
	\leq f(x),
	~
	\forall x\in S\left(A^{+},A^{-},b\right).
	$
	Therefore,
	$
	f\left(x^{*}(e^{*})\right)
	=
	\min_{x\in S(A^{+},A^{-},b)}f(x),
	$
	which means that
	$
	x^{*}(e^{*})\in S^{*}.
	$
	Thus, $x^{*}(e^{*})$ is a globally optimal clinical recommendation
	for Problem~\eqref{eq_1}.
\end{proof}

Theorem~\ref{thm-2} shows that every vector
$x^{*}(e)$, $e\in E$, is a global optimum of the clinical objective
within its corresponding clinical recommendation region $S(e)$.
Therefore, the set
$
F^{*}
=
\left\{
x^{*}(e):e\in E
\right\}
$
contains the candidate clinical recommendations obtained from all
admissible Clinical Evidence-Assignment Functions.
Consequently, Theorem~\ref{thm-3} establishes that the global
optimization of the BFRE-based CDSS can be performed by comparing the
clinical objective values of the candidate recommendations in $F^{*}$.
In particular, if
$
e^{*}\in
\operatorname*{arg\,min}_{e\in E}
f\left(x^{*}(e)\right),
$
then
$
x^{*}(e^{*})\in S^{*}.
$
Thus, the globally optimal clinical recommendation is obtained from
$
x^{*}(e^{*})
\in
\operatorname*{arg\,min}_{x\in F^{*}} f(x),
$
and the corresponding global minimum clinical objective value is
$
f\left(x^{*}(e^{*})\right)
=
\min_{e\in E}f\left(x^{*}(e)\right)
=
\min_{x\in S(A^{+},A^{-},b)}f(x).
$

From the CDSS perspective, this result has a direct interpretation.
Each admissible Clinical Evidence-Assignment Function $e$ represents a
particular admissible pattern through which clinical criteria provide
effective evidence for the treatment alternatives. The corresponding
$x^{*}(e)$ is the best clinical recommendation under that particular
evidence-assignment pattern. The set $F^{*}$ therefore collects the
best candidate recommendations generated by all admissible patterns.
The final CDSS recommendation is obtained by comparing these candidates
according to the specified clinical objective and selecting the one
with the most favorable objective value. Hence, the BFRE formulation
provides a rigorous mechanism for reducing the search for a globally
optimal clinical recommendation to the comparison of the
evidence-assignment-based candidates in $F^{*}$.

\begin{example}\label{ex-3}
To demonstrate how the proposed BFRE-based Clinical Decision Support
System operates from the characterization of clinical admissibility to
the final optimization of the clinical recommendation, we now apply
Algorithm 1 to the CDSS introduced in
Example~\ref{ex-1}. In this example, the five treatment alternatives
$T_1,\ldots,T_5$ constitute the clinical requirements of the BFRE
system, while $C_1,\ldots,C_5$ represent the clinical criteria whose
recommendation levels are determined by the decision vector
$x=(x_1,\ldots,x_5)^T$. The positive and negative relationship matrices
encode, respectively, the clinical evidence supporting and opposing
each treatment-criterion association.
	\paragraph{Step 1.}
	\label{st-1}
	
	For each treatment alternative $T_i$ and clinical criterion $C_j$,
	the Clinical Admissibility Sets $I_{ij}$ and Bipolar Evidence
	Activation Sets $S_{ij}$ are first determined from the positive and
	negative clinical relationships. The resulting sets are shown in
	Tables~\ref{t-1} and~\ref{t-2}. Here, row $i$ corresponds to treatment
	alternative $T_i$, while column $j$ corresponds to clinical criterion
	$C_j$.
	
	\begin{table}[!ht]
		\centering
		\caption{Clinical Admissibility Sets $I_{ij}$ for the CDSS example.}
		\label{t-1}
		\renewcommand{\arraystretch}{1.2}
		\begin{tabular}{c|ccccc}
			\hline
			& $C_1$ & $C_2$ & $C_3$ & $C_4$ & $C_5$\\
			\hline
			$T_1$ & $[0,0.6]$ & $[0,1]$ & $[0,1]$ & $[0,1]$ & $[0,1]$\\
			$T_2$ & $[0,0.5]$ & $[0,1]$ & $[0.5,1]$ & $[0,1]$ & $[0,1]$\\
			$T_3$ & $[0,1]$ & $[0,1]$ & $[0,1]$ & $[0,1]$ & $[0,1]$\\
			$T_4$ & $[0,0.7]$ & $[0,1]$ & $[0.3,1]$ & $[0,1]$ & $[0,1]$\\
			$T_5$ & $[0.4,1]$ & $[0,1]$ & $[0,1]$ & $[0,1]$ & $[0.4,1]$\\
			\hline
		\end{tabular}
	\end{table}
	
	\begin{table}[!ht]
		\centering
		\caption{Bipolar Evidence Activation Sets $S_{ij}$ for the CDSS example.}
		\label{t-2}
		\renewcommand{\arraystretch}{1.2}
		\begin{tabular}{c|ccccc}
			\hline
			& $C_1$ & $C_2$ & $C_3$ & $C_4$ & $C_5$\\
			\hline
			$T_1$ & $\varnothing$ & $\varnothing$ & $[0.6,1]$ & $[0,0.4]$ & $\varnothing$\\
			$T_2$ & $\{0.5\}$ & $\varnothing$ & $\{0.5\}$ & $\varnothing$ & $\varnothing$\\
			$T_3$ & $\varnothing$ & $\varnothing$ & $\varnothing$ & $[0,0.2]$ & $[0.8,1]$\\
			$T_4$ & $\varnothing$ & $[0.7,1]$ & $\varnothing$ & $\varnothing$ & $\varnothing$\\
			$T_5$ & $\{0.4\}$ & $\varnothing$ & $\varnothing$ & $\varnothing$ & $\{0.4\}$\\
			\hline
		\end{tabular}
	\end{table}
	
	By intersecting the admissibility sets corresponding to each clinical
	criterion, Definition~\ref{def-4} gives the Clinical Admissibility
	Intervals
	$I_1=[0.4,0.5]$,
	$I_2=[0,1]$,
	$I_3=[0.5,1]$,
	$I_4=[0,1]$,
	$I_5=[0.4,1].$	
	The corresponding Effective Evidence Activation Sets
	$S_{ij}^{\prime}=S_{ij}\cap I_j$ are summarized in
	Table~\ref{t-4}.
	
	\begin{table}[!ht]
		\centering
		\caption{Effective Evidence Activation Sets $S_{ij}^{\prime}$ for the CDSS example.}
		\label{t-4}
		\renewcommand{\arraystretch}{1.2}
		\begin{tabular}{c|ccccc}
			\hline
			& $C_1$ & $C_2$ & $C_3$ & $C_4$ & $C_5$\\
			\hline
			$T_1$ & $\varnothing$ & $\varnothing$
			& $[0.6,1]$ & $[0,0.4]$ & $\varnothing$\\
			
			$T_2$ & $\{0.5\}$ & $\varnothing$
			& $\{0.5\}$ & $\varnothing$ & $\varnothing$\\
			
			$T_3$ & $\varnothing$ & $\varnothing$
			& $\varnothing$ & $[0,0.2]$ & $[0.8,1]$\\
			
			$T_4$ & $\varnothing$ & $[0.7,1]$
			& $\varnothing$ & $\varnothing$ & $\varnothing$\\
			
			$T_5$ & $\{0.4\}$ & $\varnothing$
			& $\varnothing$ & $\varnothing$ & $\{0.4\}$\\
			\hline
		\end{tabular}
	\end{table}
		
	\paragraph{Step 2.}
	\label{st-2}
	
	From the above Clinical Admissibility Intervals, we have
	$
	I_j\neq\varnothing,~ \forall j\in\mathscr{J}.
	$
	Hence, the first necessary feasibility condition is satisfied: every
	clinical criterion possesses at least one recommendation level that is
	compatible with all treatment requirements.
	
	\paragraph{Step 3.}	\label{st-3}
	
	Table~\ref{t-4} shows that every treatment alternative has at least one
	clinical criterion whose Effective Evidence Activation Set is
	nonempty. More precisely,
	$\mathscr{J}_1=\{3,4\}$,
	$\mathscr{J}_2=\{1,3\}$,
	$\mathscr{J}_3=\{4,5\}$,
	$\mathscr{J}_4=\{2\}$,
	$\mathscr{J}_5=\{1,5\}.$
	Thus, each treatment requirement can be satisfied by at least one
	effective clinical criterion, and the second necessary feasibility
	condition is also satisfied. Consequently, the BFRE-based CDSS
	problem is feasible.
	
	\paragraph{Step 4.}	\label{st-4}
	For the present CDSS instance, the five simplification rules introduced
	in the paper~\cite{ref_63} do not yield any reduction of the
	BFRE system. Consequently, the original five-treatment, five-criterion
	structure is retained for the subsequent analysis.
	
	Before imposing the common-value compatibility condition associated
	with the Effective Evidence Activation Sets, the number of possible
	Clinical Evidence-Assignment Functions is bounded by
	$
	\prod_{i\in\mathscr{I}}|\mathscr{J}_i|
	=
	2\times2\times2\times1\times2
	=16.
	$
	However, not all of these assignments are admissible. Whenever two or
	more treatment requirements are assigned to the same clinical
	criterion, their corresponding Effective Evidence Activation Sets must
	have a nonempty intersection. This condition ensures the existence of
	a common recommendation level that simultaneously satisfies all
	associated bipolar clinical relationships.
	
	\paragraph{Step 5.}	\label{st-5}
	Applying the compatibility conditions in Definition~\ref{def-6} to the
	sixteen possible assignments leaves five admissible
	Clinical Evidence-Assignment Functions. They can be represented by
	the following vectors:
	$e_1=(3,1,4,2,5)$,
	$e_2=(4,1,4,2,5)$,
	$e_3=(4,3,4,2,1)$,
	$e_4=(4,3,4,2,5)$,
	$e_5=(4,3,5,2,1).$
	Here, the $i$-th component of $e_k$ specifies the clinical criterion
	assigned to the $i$-th treatment requirement.
	For example, under $e_4=(4,3,4,2,5)$, treatment requirements
	$T_1$ and $T_3$ both rely on criterion $C_4$. Since
	$
	S_{14}^{\prime}=[0,0.4]
	~\text{and}~
	S_{34}^{\prime}=[0,0.2],
	$
	their common admissible recommendation levels are
	$
	S_{14}^{\prime}\cap S_{34}^{\prime}
	=[0,0.2]\neq\varnothing.
	$
	Thus, the assignment is clinically compatible.
	In contrast, assignments that associate treatment requirements with
	the same criterion but produce disjoint Effective Evidence Activation
	Sets are excluded, since no single recommendation level could
	simultaneously activate the required bipolar clinical relationships.
	
	\paragraph{Step 6.}	\label{st-6}
	
	For each admissible assignment function $e_k$, Definition~\ref{def-7}
	gives a corresponding Clinical Requirement Satisfaction Set
	$S(e_k)$. For the present example, these sets are
	
	\[
	\begin{aligned}
		S(e_1)
		={}&
		\{0.5\}\times[0.7,1]\times[0.6,1]
		\times[0,0.2]\times\{0.4\},\\[2mm]
		S(e_2)
		={}&
		\{0.5\}\times[0.7,1]\times[0.5,1]
		\times[0,0.2]\times\{0.4\},\\[2mm]
		S(e_3)
		={}&
		\{0.4\}\times[0.7,1]\times\{0.5\}
		\times[0,0.2]\times[0.4,1],\\[2mm]
		S(e_4)
		={}&
		[0.4,0.5]\times[0.7,1]\times\{0.5\}
		\times[0,0.2]\times\{0.4\},\\[2mm]
		S(e_5)
		={}&
		\{0.4\}\times[0.7,1]\times\{0.5\}
		\times[0,0.4]\times[0.8,1].
	\end{aligned}
	\]
	
	Consequently, by Theorem~\ref{thm-1}, the complete feasible clinical
	recommendation set is
	$
	S(A^{+},A^{-},b)
	=
	\bigcup_{k=1}^{5}S(e_k).
	$
	Thus, the BFRE resolution procedure does not produce a single
	recommendation at this stage. Rather, it identifies all clinically
	admissible recommendation vectors that simultaneously satisfy the
	bipolar fuzzy relational requirements encoded in the CDSS.
	
	\paragraph{Step 7.}	\label{st-7}
	
	The final stage of the CDSS procedure is the optimization of the
	clinical recommendation objective over the feasible clinical
	recommendation set. In this example, the clinical preference/cost
	vector is specified as
	$
	c=(1,3,-1,-2,3)^{T},
	$
	and hence the clinical recommendation objective is
	$
	f(x)=c^{T}x
	=x_{1}+3x_{2}-x_{3}-2x_{4}+3x_{5}.
	$
	The coefficients describe the relative contribution of the clinical
	criteria to the overall recommendation objective. Since the problem
	is formulated as a minimization problem, a positive coefficient
	indicates that increasing the recommendation level of the corresponding
	clinical criterion increases the objective value, whereas a negative
	coefficient indicates that increasing that recommendation level
	decreases the objective value. Consequently,
	$
	\mathscr{J}^{+}=\{1,2,5\},
	~
	\mathscr{J}^{-}=\{3,4\}.
	$
	Thus, the recommendation levels associated with $C_{1}$, $C_{2}$,
	and $C_{5}$ are minimized within their admissible ranges, whereas
	those associated with $C_{3}$ and $C_{4}$ are maximized when constructing
	the candidate optimal recommendations.
	
	For each admissible Clinical Evidence-Assignment Function
	$e\in E$, Definition~\ref{def-8} therefore determines a candidate
	clinical recommendation $x^{*}(e)$ by selecting the appropriate
	lower or upper admissible levels according to the clinical preference
	structure encoded in $c$. The resulting finite collection
	$
	F^{*}
	=
	\left\{
	x^{*}(e):e\in E
	\right\}
	$
	contains all candidate region-wise optimal clinical recommendations
	generated by the admissible evidence assignments. By
	Theorem~\ref{thm-2}, every
	$x^{*}(e)\in F^{*}$ is a feasible locally optimal solution.
	
	The global clinical recommendation is then obtained by comparing the
	clinical objective values of these candidate recommendations. In
	particular, let
	$
	e^{*}\in
	\operatorname*{arg\,min}_{e\in E}
	f\left(x^{*}(e)\right).
	$
	Then, by Theorem~\ref{thm-3},
	$
	x^{*}(e^{*})
	$
	is a globally optimal clinical recommendation, and the corresponding
	global clinical preference value is
	$
	f\left(x^{*}(e^{*})\right)
	=
	\min_{e\in E}
	f\left(x^{*}(e)\right).
	$
	
	Accordingly, the CDSS does not merely identify recommendation vectors
	that satisfy the bipolar clinical requirements. It first uses the
	BFRE constraints to determine the clinically feasible recommendation
	space and then uses the clinical preference/cost vector $c$ to select
	the most preferred feasible recommendation. In this way, the BFRE
	system provides the underlying mathematical mechanism for translating
	bipolar clinical relationships into admissible recommendation levels,
	while the objective function expresses the CDSS preference among those
	clinically admissible alternatives.
	
	Consequently, the five admissible evidence-assignment functions generate
	the candidate clinical recommendations
	$$
	\begin{aligned}
		x^{*}(e_1)&=(0.5,0.7,1,0.2,0.4),&
		f(x^{*}(e_1))&=2.4,\\
		x^{*}(e_2)&=(0.5,0.7,1,0.2,0.4),&
		f(x^{*}(e_2))&=2.4,\\
		x^{*}(e_3)&=(0.4,0.7,0.5,0.2,0.4),&
		f(x^{*}(e_3))&=2.8,\\
		x^{*}(e_4)&=(0.4,0.7,0.5,0.2,0.4),&
		f(x^{*}(e_4))&=2.8,\\
		x^{*}(e_5)&=(0.4,0.7,0.5,0.4,0.8),&
		f(x^{*}(e_5))&=3.6.
	\end{aligned}
	$$
	Thus,
	$
	\min_{e\in E} f(x^{*}(e))=2.4,
	$
	attained for $e_1$ and $e_2$. Since
	$
	x^{*}(e_1)=x^{*}(e_2)
	=(0.5,0.7,1,0.2,0.4),
	$
	the resulting globally optimal clinical recommendation is
	$x^{*}=(0.5,0.7,1,0.2,0.4),
	$
	with global minimum clinical objective value
	$f(x^{*})=2.4.
	$
	
\end{example}
\section*{Conclusion}\label{sec_con}
This paper developed a bipolar fuzzy relational optimization framework for clinical decision support in which positive and negative clinical relationships are incorporated simultaneously within a unified recommendation process. The clinical information entering the model is represented through the fuzzy relational matrices \(A^{+}\) and \(A^{-}\) together with the fuzzy requirement vector \(b\), providing a compact graded representation of the relational knowledge on which the BFRE constraints operate.

For the max--min formulation, Clinical Admissibility Intervals and Effective Evidence Activation Sets were introduced to provide a componentwise characterization of feasibility. Building on these structures, the complete feasible recommendation set was represented as a finite union of regions induced by admissible Clinical Evidence-Assignment Functions. This representation makes it possible to describe the potentially nonconvex BFRE solution set through a finite family of clinically interpretable recommendation regions. Under the assumed continuity and coordinatewise monotonicity of the objective function, a region-wise optimal candidate can be constructed for each admissible evidence assignment, and comparison of the resulting finite collection of candidates yields a globally optimal recommendation. These results were integrated into an algorithm that combines feasibility analysis, structural reduction, construction of the feasible region, region-wise optimization, and global selection.

The numerical example illustrated the complete decision process, from bipolar fuzzy relational data to the characterization of admissible recommendation regions and the identification of a globally optimal relational decision vector. More broadly, the proposed framework shows how BFRE optimization can provide a rigorous mathematical mechanism for decision-support problems in which graded favorable and unfavorable relationships must be considered simultaneously. The finite-region representation and evidence-assignment structure further provide a transparent description of how feasible and preferred recommendations are generated from the underlying bipolar relational information. These properties provide a natural foundation for further application-oriented investigations involving clinically derived fuzzy relational data, larger-scale decision-support instances, and systematic analyses of the robustness and sensitivity of the resulting recommendations.

\printbibliography
\end{document}